\documentclass[12pt]{article}%
\usepackage{graphicx}
\usepackage{pgf,tikz}
\usepackage{amsmath,amssymb}
\usepackage{amsmath}
\usepackage{amsfonts}
\usepackage{amssymb}%
\newcommand{\comment}[1]{}

\newif\ifpdf
\ifpdf \else \fi \textwidth = 6.5 in \textheight = 9 in
\newtheorem{theorem}{Theorem}[section]

\newtheorem{cor}[theorem]{Corollary}

\newtheorem{definition}[theorem]{Definition}

\newtheorem{lem}[theorem]{Lemma}
\newtheorem{obs}[theorem]{Observation}
\newtheorem{notation}[theorem]{Notation}

\newenvironment{proof}[1][Proof]{\textbf{#1.} }{\ \rule{0.5em}{0.5em}}

\newcommand{\ZZ}{\mathbb{Z}}

\title{On the Directed Oberwolfach Problem \\ with variable cycle lengths}
\author{Elaheh Shabani
\footnote{Department of Mathematics, Shahrood University of Technology, Shahrood, Iran}
\\ {\small Shahrood University of Technology} \\ \\
Mateja \v{S}ajna \\
{\small University of Ottawa}}

\begin{document}

\maketitle

\begin{abstract}
The Directed Oberwolfach Problem can be considered as the directed version of the well-known Oberwolfach Problem, first mentioned by Ringel at a conference in Oberwolfach, Germany in 1967. In this paper, we describe some new partial results on the Directed Oberwolfach Problem with variable cycle lengths. In particular, we show that the complete symmetric digraph $K_n^{*}$ admits a $(\vec{C}_2,...,\vec{C}_2, \vec{C}_3)$-factorization for all $n \equiv 1, 3,$ or $7 \pmod {8}$. We also show that $K^{*}_{n}$ admits a $(\vec{C}_2, \vec{C}_{n-2})$-factorization for any integer $n \geq 5$.

\end{abstract}

{\bf Keywords:} Directed Oberwolfach Problem; complete symmetric digraph; directed 2-factorization.

\section{Introduction}

The Directed Oberwolfach Problem is the directed version of the well-known Oberwolfach Problem, first mentioned by Ringel at a conference in Oberwolfach, Germany, in 1967. At a conference in Oberwolfach, assume $n$ participants are to be seated around circular tables of specified sizes for $\frac{n-1}{2}$ consecutive nights, where the total number of seats is equal
to $n$ and $n$ is odd. The Oberwolfach Problem asks whether it is possible that each participant sits next to each other participant exactly once. When tables are of sizes $m_1, ..., m_t$, the problem is denoted by $OP(m_1, ..., m_t)$. For $n$ even, the analogous problem is called the Spouse-Avoiding Variant, and is described as follows. Assume $n=2k$ participants, consisting of $k$ couples, are to be seated around $t$ tables of sizes $m_1, ..., m_t$ for $k-1$ consecutive nights, where $m_1+ ...+ m_t=n$. The Spouse-Avoiding Variant asks whether it is possible that each participant sits next to each other participant exactly once, except they never sit next to their spouse.

The original Oberwolfach Problem can be modeled as a decomposition of the complete graph into isomorphic 2-factors. Although many cases of this problem have been solved since 1967, the problem in general is still open.

The Oberwolfach Problem with uniform cycle lengths has been solved completely \cite{AH, ASSW, HS,RW}. In these papers it was shown that $OP(m, m, ..., m)$ has a solution for each $n$ and $m$ such that $m$ divides $n$, except when $n\in \{6, 12\}$ and $m=3$. The Oberwolfach Problem has been solved partially when the cycles have variable lengths. Bryant and Danziger \cite{BD} have shown that $OP(m_1, m_2, ..., m_t)$ has a solution for all $n$ and $m_1, m_2, ..., m_t$ all even. Traetta \cite{TT} has proved that the Oberwolfach Problem for two tables has a solution except for the case with two tables of size $3$, and the case with a table of size $4$ and a table of size $5$. Other authors \cite{ad, DFW} have shown that $OP(m_1, m_2, ..., m_t)$ has a solution for all $n\leq 40$ except for $OP(3, 3)$, $OP(3, 3, 3, 3)$, $OP(4, 5)$, and $OP(3, 3, 5)$ which are the only known exceptions. Recently, Glock et al \cite{GJ} proved that the Oberwolfach Problem has a solution for all large $n$.

The directed version would then be asking whether it is possible that each participant sits to \textit{the right} of each other participant exactly once. The Directed Oberwolfach Problem when tables are of sizes $m_1, ..., m_t$ and $m_1+ ...+m_t=n$ is denoted by $OP^{*}(m_1, ..., m_t)$.
The Directed Oberwolfach Problem has been solved in the case of cycles of length $3$ by Bermond, Germa, and Sotteau \cite{BGDR}, and in the case of cycles of length $4$ by Bennett and Zhang \cite{BZ}, except for one missing case $(n = 12)$ filled in by Adams and Bryant \cite{AB}.

The following theorems summarize all previous results on the Directed Oberwolfach Problem.
\begin{theorem}\cite{BGDR}
$OP^{*}(3, 3, ..., 3)$ has a solution if and only if $3$ divides $n$ and $n \neq6$.
\end{theorem}

\begin{theorem}\cite{AB, BZ}
$OP^{*}(4, 4, ..., 4)$ has a solution if and only if $4$ divides $n$ and $n \neq 4$.
\end{theorem}

\begin{theorem}\cite{BM} Let $m$ and $n$ be integers with $5 \leq m \leq n$. Then the following hold.

1. Let $m$ be even, or $m$ and $n$ be both odd. Then $OP^{*}(m, m, ..., m)$ has a solution if and only if $m$ divides $n$ and $(m, n) \neq(6, 6)$.

2. If $OP^{*}(m, m)$ has a solution, then $OP^{*}(m, m, ..., m)$ has a solution whenever $n \equiv 0 \pmod {2m}$.
\end{theorem}

\begin{theorem}\cite{BFN}
Let $m$ be an odd integer, $5 \leq m \leq 49$. Then $OP^{*}(m, m, ..., m)$ has a solution whenever $n \equiv 0 \pmod {2m}$.
\end{theorem}

 In this paper, we will describe some new partial results on the Directed Oberwolfach Problem with variable cycle lengths. 
The paper is organized as follows. In Section 2, we introduce the terminology and
present some basic observations on the problem. In Section 3, we prove that $OP^{*}(4, 5)$ and $OP^{*}(3, 3, 5)$ have a solution. We also show that $OP^{*}(2, n-2)$ has a solution for all $n\geq 5$. Consequently, for $2\leq m \leq n-2$ and $n$ odd, $OP^{*}(m, n-m)$ has a solution if and only if $(m, n)\neq (3, 6)$. Finally, we prove that $OP^{*}(2, 2, ..., 2, 3)$ has a solution whenever $n \equiv 1, 3,$ or $7 \pmod {8}$.
\section{Prerequisites}

As usual, the symbol $K_n^\ast$ denotes the complete symmetric digraph with $n$ vertices.
A \textit{decomposition} of a graph $G$ is a set $\mathcal{D}=\{G_1, G_2,...,G_t\}$ of subgraphs of $G$ such that $E(G_1)\cup E(G_2)\cup...\cup E(G_t)=E(G)$ and $E(G_i)\cap E(G_j)=\emptyset$ for $i\neq j$. In the case where the subgraphs are cycles we have a \textit{cycle decomposition}. 

In this paper, we consider the problem of decomposing the complete symmetric digraph $K^{*}_{n}$ into directed 2-factors, which we define as follows. Let $D$ be a digraph with $n$ vertices. For $2 \le m \le n$, a directed cycle of length $m$ in $D$ is denoted by $\vec{C}_m$. A {\em directed 2-factor} of $D$ is a spanning subdigraph of $D$ that is a disjoint union of directed cycles. For integers $m_1,\ldots,m_t$ such that $2 \le m_1 \le \ldots \le m_t$ and $m_1+\ldots+m_t=n$, a {\em $(\vec{C}_{m_1},\ldots,\vec{C}_{m_t})$-factor} of $D$ is a directed 2-factor of $D$ consisting of $t$ pairwise disjoint directed cycles of lengths $m_1,\ldots,m_t$, respectively. A {\em $(\vec{C}_{m_1},\ldots,\vec{C}_{m_t})$-factorization} of $D$ is a decomposition of $D$ into $(\vec{C}_{m_1},\ldots,\vec{C}_{m_t})$-factors.

\begin{definition} {\rm
Let $n$ be a positive integer and $S \subseteq \ZZ^\ast$. The {\em circulant digraph} $\vec{X}(n; S)$ with connection set $S$ is the digraph whose vertex set is $\{ u_i: i \in \ZZ_n \}$,  with an arc from $u_i$ to $u_j$ if and only if $j-i \in S$. Such an arc $(u_i,u_j)$ will be called of {\em difference} $j-i$ (evaluated in $\ZZ_n$).
}
\end{definition}

\begin{definition}{\rm
If $D_1$ and $D_2$ are vertex-disjoint digraphs, then $D_1 \bowtie D_2$ is the digraph obtained by taking the union of $D_1$ and $D_2$
together with all possible arcs from $D_1$ to $D_2$ and from $D_2$ to $D_1$.
}
\end{definition}

\begin{definition}{\rm
Let $G$ be a graph. Then we say that $S\subseteq E(G)$ {\em covers} a vertex $x\in V(G)$ $t$ times if the vertex $x$ is incident to  exactly $t$ edges of $S$.
}
\end{definition}

In this paper we examine the existance of a $(\vec{C}_{m_1}, ..., \vec{C}_{m_t})$-factorization of $K_{n}^{*}$. An obvious necessary condition is that $2\leq m_i \leq n$, for each $i=1,2,..., t$, and also $m_1+m_2+...+m_t=n$.

We can easily see that if there exists a solution for the Oberwolfach Problem (for $n$ odd) with cycles of lengths $m_1, ..., m_t$, then there exist a solution for the Directed Oberwolfach Problem with cycles of lengths $m_1, ..., m_t$.

\begin{obs}\label{obs}
If $OP(m_1, m_2, ..., m_t)$ has a solution and $n$ is odd, then $OP^{*}(m_1, m_2, ..., m_t)$ has a solution.
\end{obs}

A solution for $OP^{*}(m_1, m_2, ..., m_t)$ is obtained from a solution of $OP(m_1, m_2, ..., m_t)$ by taking two copies of  each 2-factor, and directing each cycle in the two 2-factors in two ways.

\section{Results}

In this section, we find a solution for some cases of the Directed Oberwolfach Problem that cannot be solved using Observation \ref{obs}. First, we consider two cases of the Directed Oberwolfach Problem where the length of each cycle is more than two. These correspond to the only known cases of the Oberwolfach Problem with variable cycle lengths that are known to have no solution. Then, we examine some cases of the Directed Oberwolfach Problem which have at least one cycle of length two.  

\begin{lem}\label{lmm:C4+C5}
There exists a $(\vec{C}_4, \vec{C}_{5})$-factorization of $K_9^*$.
\end{lem}
\begin{proof}
View $K_9^*$ as the join $K_{8}^* \bowtie K_1^*$, where the vertex set of $K_1^*$ is $\{ u_{\infty} \}$, and $K_{8}^*$ is viewed as the circulant digraph $\vec{X}(8;S)$ with vertex set $\{ u_i: i \in \ZZ_{8} \}$ and connection set $S=\{ \pm 1,\pm 2,\pm 3,4 \}$. For $i \in \ZZ_{8}$, arcs of the forms $(u_i,u_{\infty})$ and $(u_{\infty}, u_i)$ will be called of difference $\infty$ and $-\infty$, respectively. Define the permutation $\rho=(u_0 \, u_1 \, \ldots \, u_{7})(u_{\infty})$.

Next define the directed 5-cycle 
$$C_0=u_1 \, u_2 \, u_{\infty} \, u_6 \, u_4 \, u_1,$$
and the directed 4-cycle 
$$C_1=u_0 \, u_7 \, u_3 \, u_5 \, u_0.$$

Observe that $R=\{C_0, C_1\}$ is a $(\vec{C}_4, \vec{C}_{5})$-factor of $K_9^*$ (see Figure \ref{45}) containing exactly one arc of each difference in the set
$$\left\{ \pm 1, \pm 2, \pm3, 4, \pm \infty \right\}.$$\\

\begin{figure}[h!]
\begin{center}
\centerline{\includegraphics[scale=0.35]{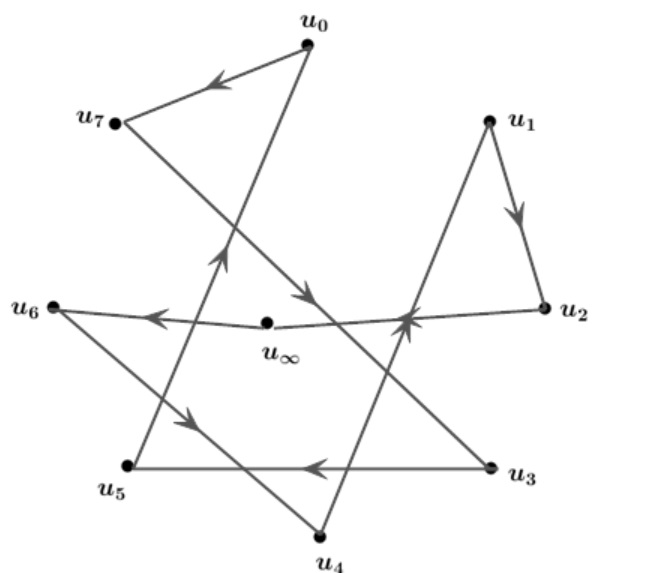}}
\caption{The $(\vec{C}_4, \vec{C}_5)-$factor $R$ of $K^{*}_{9}$.}
\label{45}
\end{center}
\end{figure}

From the properties of $R$ we conclude that
$$\{ \rho^i(R): i \in \ZZ_{8} \}$$
is a $(\vec{C}_4, \vec{C}_{5})$-factorization of $K_9
^*$. 
\hfill
\end{proof}


\begin{lem}\label{lmm:C3+C3+C5}
There exists a $(\vec{C}_3, \vec{C}_3, \vec{C}_{5})$-factorization of $K_{11}^*$.
\end{lem}

\begin{proof}
View $K_{11}^*$ as the join $K_{10}^* \bowtie K_1^*$, where the vertex set of $K_1^*$ is $\{ u_{\infty} \}$, and $K_{10}^*$ is viewed as the circulant digraph $\vec{X}(10;S)$ with vertex set $\{ u_i: i \in \ZZ_{10} \}$ and connection set $S=\{ \pm 1,\pm 2,\pm 3,\pm4, 5 \}$. For $i \in \ZZ_{10}$, arcs of the forms $(u_i,u_{\infty})$ and $(u_{\infty}, u_i)$ will be called of difference $\infty$ and $-\infty$, respectively. Define the permutation $\rho=(u_0 \, u_1 \, \ldots \, u_{9})(u_{\infty})$.
Next define the two directed 3-cycles 
$$C_0 = u_{0} \, u_{1} \, u_{3} \, u_{0} \quad \mbox{ and } \quad C _1= u_{2} \, u_{7} \, u_{\infty} \, u_{2},$$
and the directed 5-cycle 
$$C_2=u_4 \, u_8 \, u_6 \, u_9 \, u_5 \, u_{4}.$$
Observe that $R=\{C_0, C_1, C_2\}$ is a $(\vec{C}_3, \vec{C}_3, \vec{C}_{5})$-factor of $K_{11}^*$ containing exactly one arc of each difference in the set
$$ \Big\{ \pm 1, \pm 2, \pm 3, \pm 4, 5, \pm \infty \Big\}. $$

\begin{figure}[h!]
\begin{center}

\centerline{\includegraphics[scale=0.35]{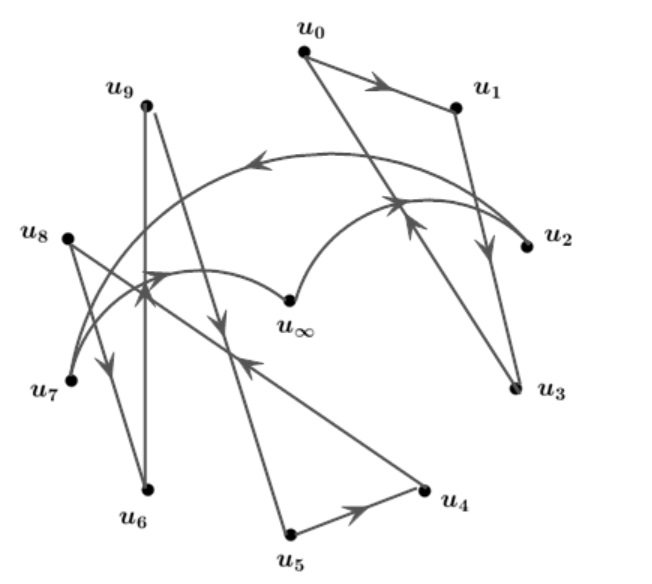}}
\caption{The $(\vec{C}_3,\vec{C}_3, \vec{C}_5)$-factor $R$ of $K^{*}_{11}$.}
\end{center}
\end{figure}

From the properties of $R$ we conclude that
$$\Big\{ \rho^i(R): i \in \ZZ_{10} \Big\}$$
is a $(\vec{C}_3, \vec{C}_3, \vec{C}_{5})$-factorization of $K_{11}
^*$. 
\hfill
\end{proof}

\begin{theorem}\label{thm mod 8}
Let $n$ be an odd integer such that $7\leq n$ and also $n \equiv i (\mod 8
)$ where $i=1, 3$. Then there exsits a $(^{\frac{n-3}{2}}\vec{C}_2, \vec{C}_{3})$-factorization of $K_n^*$.
\end{theorem}


\begin{theorem}\label{thm:C2+C(n-2)}
For any integer $n \geq 5$, the digraph $K^{*}_{n}$ admits a $(\vec{C}_2, \vec{C}_{n-2})$-factorization.
\end{theorem}

\begin{proof}
{\sc Case $n$ odd, $n \ge 7$.} Let $\ell=\frac{n-3}{2}$, and view $K_n^*$ as the join $K_{2\ell+1}^* \bowtie K_2^*$, where the vertex set of $K_2^*$ is $\{ u_{\infty_1},u_{\infty_2} \}$, and $K_{2\ell+1}^*$ is further viewed as the circulant digraph $\vec{X}(2\ell+1;S)$ with vertex set $\{ u_i: i \in \ZZ_{2\ell+1} \}$ and connection set $S=\{ \pm d: d=1,2,\ldots,\ell \}$. For $i \in \ZZ_{2\ell+1}$ and $j \in \{ 1,2 \}$, arcs of the forms $(u_i,u_{\infty_j})$ and $(u_{\infty_j}, u_i)$ will be called of difference $\infty_j$ and $-\infty_j$, respectively. The difference of the arcs  $(u_{\infty_1},u_{\infty_2})$ and $(u_{\infty_2},u_{\infty_1})$ remains undefined. Define the permutation $\rho=(u_0 \, u_1 \, \ldots \, u_{2\ell})(u_{\infty_1})(u_{\infty_2})$.

Let $k=\lfloor \frac{\ell}{2} \rfloor$, and define the directed closed walk
$$C = u_0 \, u_{-1} \, u_{1}\, \ldots \, u_{-k} \, u_{k}\, u_{-(k+2)} \, u_{k+1}\, \ldots \, u_{-\ell} \, u_{\ell-1} \, u_{\ell} \, u_{\infty_2} \, u_0.$$
It is not difficult to verify that $C$ is in fact a directed $(2\ell+1)$-cycle containing exactly one arc of each difference in the set
$$\{ \pm 1, \pm 2, \ldots, \pm (\ell-1), \ell, \pm \infty_1 \}.$$
Moreover, $C$ is disjoint from the directed 2-cycle
$$C'=u_{-(k+1)} \, u_{\infty_1} \, u_{-(k+1)},$$
and hence $R=\{ C, C' \}$ is a $(\vec{C}_2, \vec{C}_{n-2})$-factor of $K_n^*$ containing exactly one arc of each difference in the set
$$D=\{ \pm 1, \pm 2, \ldots, \pm (\ell-1), \ell, \pm \infty_1, \pm \infty_2 \}.$$
Next, we define the directed $(2\ell+1)$-cycle
$$C'' = u_0 \, u_{-\ell} \, u_1 \, u_{-(\ell-1)} \, u_2 \, \ldots \, u_{-2} \, u_{\ell-1} \, u_{-1} \,  u_{\ell} \, u_0$$
and the directed 2-cycle
$$C''' = u_{\infty_1} \, u_{\infty_2} \, u_{\infty_1}.$$
Observe that $R'=\{ C'', C''' \}$ is a $(\vec{C}_2, \vec{C}_{n-2})$-factor of $K_n^*$ containing all arcs of difference $-\ell$, as well as the two arcs of undefined difference. For an example, you can see the $(\vec{C}_2, \vec{C_7})$-factors $R$ and $R^{'}$ of $K^{*}_9$ in Figure \ref{rrr}.
\\
\begin{figure}[h!]\label{rrr}
\begin{center}

\centerline{\includegraphics[scale=.52]{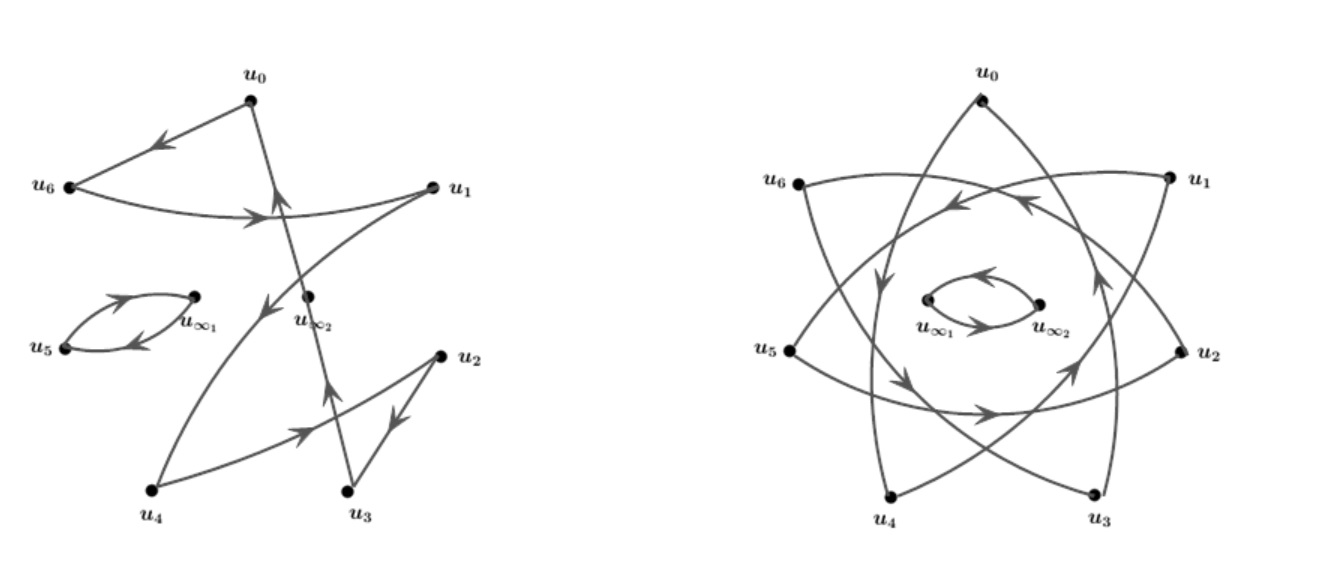}}
\caption{The $(\vec{C}_2, \vec{C_7})$-factors $R$ and $R^{'}$ of $K^{*}_9$.}
\end{center}
\end{figure}
\\
From the properties of $R$ and $R'$ we conclude that
$$\{ \rho^i(R): i \in \ZZ_{2\ell+1} \} \cup \{ R' \}$$
is a $(\vec{C}_2, \vec{C}_{n-2})$-factorization of $K_n^*$.

{\sc Case $n$ even or $n = 5$.} Now view $K_n^*$ as the join $K_{n-1}^* \bowtie K_1^*$, where the vertex set of $K_1^*$ is $\{ u_{\infty} \}$, and $K_{n-1}^*$ is  viewed as the circulant digraph $\vec{X}(n-1;S)$ with vertex set $\{ u_i: i \in \ZZ_{n-1} \}$ and connection set $S=\ZZ_{n-1}^*$. For $i \in \ZZ_{n-1}$, arcs of the forms $(u_i,u_{\infty})$ and $(u_{\infty}, u_i)$ will be called of difference $\infty$ and $-\infty$, respectively. Define the permutation $\rho=(u_0 \, u_1 \, \ldots \, u_{n-2})(u_{\infty})$.

First assume $n=5$. We define two directed 3-cycles
$$C_0 = u_{0} \, u_{1} \, u_{2} \, u_{0} \quad \mbox{ and } \quad C _1= u_{3} \, u_{2} \, u_{1} \, u_{3},$$
and two directed 2-cycles
$$C'_0 = u_{3} \, u_{\infty} \, u_{3}  \quad \mbox{ and } \quad C'_1 = u_{0} \, u_{\infty} \, u_{0}.$$
Observe that each of $R_0=\{C_0, C'_0\}$ and $R_1=\{C_1, C'_1\}$ is a $(\vec{C}_2, \vec{C}_{3})$-factor of $K_5^*$. Let $R_2$ and $R_3$ be obtained from $R_0$ and $R_1$, respectively, by adding 2 to the subscript of each vertex. It is not difficult to verify that $\{ R_0, R_1, R_2, R_3 \}$ is a $(\vec{C}_2, \vec{C}_{3})$-factorization of $K_5^*$.

Next, assume $n$ is even. Let $k=\lceil\frac{n-4}{4}\rceil$, and define the directed closed walk
$$C_0 = u_0 \, u_{-1} \, u_{1}\, \ldots \, u_{-k} \, u_{k}\, u_{-(k+2)} \, u_{k+1}\, \ldots \, u_{-\frac{n-2}{2}} \, u_{\frac{n-4}{2}} \, u_{\frac{n-2}{2}} \, u_0.$$
It is not difficult to verify that $C_0$ is in fact a directed $(n-2)$-cycle containing exactly one arc of each difference in the set 
$$\left\{ \pm 1, \pm 2, \ldots, \pm \frac{n-2}{2} \right\}.$$
Moreover, $C_0$ is disjoint from the directed 2-cycle
$$C_1=u_{-(k+1)} \, u_{\infty} \, u_{-(k+1)},$$
and hence $R=\{ C_0, C_1 \}$ is a $(\vec{C}_2, \vec{C}_{n-2})$-factor of $K_n^*$ containing exactly one arc of each difference in the set
$$\left\{ \pm 1, \pm 2, \ldots, \pm \frac{n-2}{2}, \pm \infty \right\}=\ZZ_{n-1}^* \cup \{ \pm \infty \}.$$
For an example, you can see the $(\vec{C}_2, \vec{C}_{8})$-factor $R$ of $K_{10}^*$ in Figure \ref{h}.
\\
\begin{figure}[h!]
\begin{center}

\centerline{\includegraphics[scale=.52]{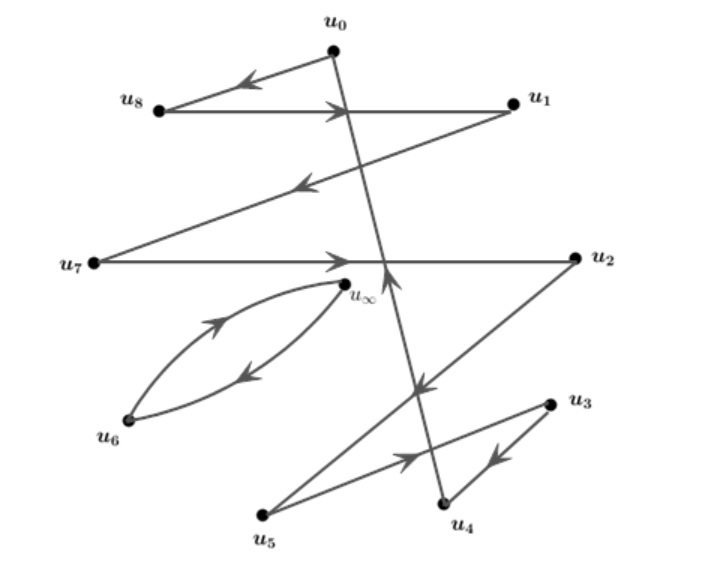}}
\caption{A $(\vec{C}_2, \vec{C}_{8})$-factor of $K_{10}^*$.}
 \label{h}
\end{center}
\end{figure}
\\
From the properties of $R$ we conclude that
$$\{ \rho^i(R): i \in \ZZ_{n-1} \}$$
is a $(\vec{C}_2, \vec{C}_{n-2})$-factorization of $K_n^*$. 
\hfill
\end{proof}


As a consequence of Theorem \ref{thm:C2+C(n-2)}, we obtain a complete solution to the Directed Oberwolfach Problem with two tables and $n$ odd.
\begin{cor}
Let $2\leq m \leq n-2$ and $n$ be odd. Then $K_{n}^{*}$ admits a $(\vec{C}_{m}, \vec{C}_{n-m})$-factorization if and only if $(m, n)\neq (3, 6)$.
\end{cor}
\begin{proof}

If $m=2$, then from Theorem \ref{thm:C2+C(n-2)}, obviously $K_{n}^{*}$ admits a $(\vec{C}_{m}, \vec{C}_{n-m})$-factorization.
Suppose that $m\geq 3$. Traetta \cite{TT} has proved that the Oberwolfach Problem for two tables has a solution except for the case with two tables of size $3$, and the case with a table of size $4$ and a table of size $5$. Thus, $OP(m, n-m)$ has a solution except for the cases $(m, n)=(3, 6)$ and $(4, 9)$. Then, Observation \ref{obs} and Lemma \ref{lmm:C4+C5} yield that  $OP^{*}(m, n-m)$ has a solution if and only if $(m, n)\neq (3, 6)$. 

\end{proof}


Now, we consider the Directed Oberwolfach Problem in the case of $t$ cycles of length 2 and one cycle of length 3, where $n=2t+3$. We proceed with the following notation and lemmas. 

\begin{notation}\rm\label{note1}Let $k$ be a positive integer and $L=\{1, 2, ..., \lfloor \frac{k}{2} \rfloor\}$. 
Take $\{x_i: i\in \ZZ_{k}\}\cup \{y_i: i\in \ZZ_{k}\}$ as the vertex set of $K_{2k}$. For $i \in \ZZ_{k}$ and $ d \in L$, edges of the forms $x_ix_{i+d}$ and $y_{i}y_{i+d}$ will be called, respectively, edges of \textit{left} and \textit{right pure length} $d$. Edges of the form $x_i y_{i+d}$, where $ i,d \in \ZZ_{k}$, will be called edges of\textit{ mixed difference} $d$. 

Let $S$ be a subset of the edge set of $K_{2k}$. Then, we define sets $\mathcal{L}(S), \mathcal{R}(S), \mathcal{M}(S), \mathcal{X}(S),$ and $ \mathcal{Y}(S)$ as follows.
\begin{eqnarray*}
\mathcal{L}(S)&=&\Big\{d\in L: x_{i}x_{i+d} \in S,~ \mbox{for some } i\in \ZZ_{k}  \Big\},\\
\mathcal{R}(S)&=&\Big\{d\in L: y_{i}y_{i+d} \in S,~ \mbox{for some } i\in \ZZ_{k}  \Big\},\\
\mathcal{M}(S)&=&\Big\{d\in \ZZ_{k}: x_{i}y_{i+d} \in S,~ \mbox{for some } i\in \ZZ_{k}  \Big\},\\
\mathcal{X}(S)&=& \Big\{ i\in \ZZ_{k}: x_iu\in S , ~ \mbox{for some } u \in V( K_{2k} ) \Big\},\\
\mathcal{Y}(S)&=& \Big\{ i\in \ZZ_{k}: y_iu\in S , ~ \mbox{for some } u \in V( K_{2k} ) \Big\}.\\
\end{eqnarray*}

\end{notation}\rm


\begin{lem}\label{lmm:k_{4ell}}
Let $\ell$ be a positive integer, and adopt the terminology from Notation \ref{note1}. Assume $K_{4\ell}$ has 1-factors $F_1$ and $F_2$ such that

(i) $F_1$ and $F_2$ jointly contain exactly one edge of each $($left and right$)$ pure length, and exactly one edge of each mixed difference, and 

(ii) $F_1$ and $F_2$ each contain exactly one edge of pure length $\ell$.
 
 Then $K^{*}_{4\ell+1}$ admits a $(\vec{C}_2,...,\vec{C}_2, \vec{C}_3)$-factorization.
\end{lem}
\begin{proof}

View $K_{4\ell+1}^*$ as the join $K_{4\ell}^* \bowtie K_1^*$, where the vertex set of $K_1^*$ is $\{ u_{\infty} \}$, and $V(K_{4\ell}^*)=V(K_{4\ell})$. Arcs of the forms $(x_i, x_{i+d})$ and $(y_i, y_{i+d})$, for $d=1,..., 2\ell-1,$ will be called arcs of left and right, respectively, pure difference $d$. Arcs of the forms $(x_i, y_{i+d})$ and $(y_{i+d}, x_{i})$, for $d=0, 1,..., 2\ell-1,$ will be called arcs of left and right, respectively, mixed difference $d$. Also arcs of the forms $(x_i,u_{\infty})$ and $(y_i,u_{\infty})$ will be called of left and right, respectively, difference $\infty$, and arcs of the forms $(u_{\infty},x_i)$ and $(u_{\infty},y_i)$ will be called of left and right, respectively, difference $-\infty$. Define the permutation $\rho=(x_0 x_1...x_{2\ell-1})(y_0 y_1...y_{2\ell-1})(u_{\infty}).$

Now, assume that $F_1$ and $F_2$ are 1-factors of $K_{4\ell}$ which satisfy both of the assumptions $(i)$ and $(ii)$. 
Let $F_{1}^{'}$ and $F_{2}^{'}$ be obtained from $F_1$ and $F_2$ respectively, by replacing each edge $uv$, except edges of pure length $\ell$, by the directed 2-cycle $u\, v\, u$, and replacing 
each edge $uv$ of pure length $\ell$ by the directed 3-cycle $u\, v\, u_{\infty}\, u$. Since $F_1$ and $F_2$ are 1-factors of $K_{4\ell}$, and each of $F_1$ and $F_2$ contains exactly one edge of pure length $\ell$, it is clear that $F^{'}_1$ and $F^{'}_2$ are $(\vec{C}_2,...,\vec{C}_2, \vec{C}_3)$-factors of $K^{*}_{4\ell+1}$. Morever, observe that each edge of pure length $d$ in $K_{4\ell}$, where $d\neq \ell$, was replaced by two opposite arcs with the same endpoints; that is, arcs of pure differences $d$ and $-d$. Similarly, each edge of mixed difference $d$ was replaced by two opposite arcs with the same endpoints; that is, arcs left and right mixed difference $d$. Finally, the edge of pure left (right ) length $\ell$ gave rise to one arc of pure left (right) difference $\ell$, one arc of left (right) difference $\infty$, and one arc of left (right) difference $-\infty$.
Thus, jointly, $F_{1}^{'}$ and $F_{2}^{'}$ contain exactly one arc of each left and right pure difference in the set
$$ \Big\{ 1, ... , 2\ell-1 \Big\}, $$
 exactly one arc of each left and right mixed difference in the set
 $$\Big\{ 0,  1, ..., 2\ell-1\Big\},$$
and exactly one arc of each left and right difference $\pm\infty.$

From the properties of $F_{1}^{'}$ and $F_{2}^{'}$ we conclude that
$$\Big\{ \rho^i(F_{1}^{'}): i \in \ZZ_{2\ell} \Big\}\cup \Big\{ \rho^i(F_{2}^{'}): i \in \ZZ_{2\ell} \Big\}$$
is a  $(\vec{C}_2,...,\vec{C}_2, \vec{C}_{3})$-factorization of $K_{4\ell+1}^{*}$.

\end{proof}


\begin{lem}\label{lmm:k_{4ell+2}}
Let $\ell$ be a positive integer, and adopt the terminology from Notation \ref{note1}. Assume $K_{4\ell +2}$ has 1-factors $F_1$ and $F_2$ such that

(i) $F_1$ and $F_2$ jointly contain exactly one edge of each (left and right) pure length and exactly one edge of each mixed difference, except mixed difference $0$, and 

(ii) $F_1$ and $F_2$ each contain exactly one edge of mixed difference $0$.

Then $K^{*}_{4\ell+3}$ admits a $(\vec{C}_2,...,\vec{C}_2, \vec{C}_3)$-factorization.

\end{lem}
\begin{proof}
View $K_{4\ell+3}^*$ as the join $K_{4\ell+2}^* \bowtie K_1^*$, where the vertex set of $K_1^*$ is $\{ u_{\infty} \}$, and $V(K_{4\ell+2}^*)=V(K_{4\ell+2})$. Arcs of the forms $(x_i, x_{i+d})$ and $(y_i, y_{i+d})$, for $d=1,..., 2\ell,$ will be called arcs of left and right, respectively, pure difference $d$. Arcs of the forms $(x_i, y_{i+d})$ and $(y_{i+d}, x_{i})$, for $d=0, 1,..., 2\ell,$ will be called arcs of left and right, respectively, mixed difference $d$. Also arcs of the forms $(x_i,u_{\infty})$ and $(y_i,u_{\infty})$ will be called of left and right, respectively, difference $\infty$, and arcs of the forms $(u_{\infty},x_i)$ and $(u_{\infty},y_i)$ will be called of left and right, respectively, difference $-\infty$. Define the permutation $\rho=(x_0 x_1...x_{2\ell})(y_0 y_1...y_{2\ell})(u_{\infty}).$

Now, assume that $F_1$ and $F_2$ are 1-factors of $K_{4\ell+2}$ which satisfy both of the assumptions $(i)$ and $(ii)$. 
Let $F_{1}^{'}$ and $F_{2}^{'}$ be obtained from $F_1$ and $F_2$ respectively, by replacing each edge $uv$, except the edges of mixed difference $0$, by the directed 2-cycle $u\, v\, u$, replacing 
one edge of mixed difference $0$ by the directed 3-cycle $u\, v\, u_{\infty}\, u$, and the other by the directed 3-cycle $v\, u\, u_{\infty}\, v$. Since $F_1$ and $F_2$ are 1-factors of $K_{4\ell+2}$, and each of $F_1$ and $F_2$ contains exactly one edge of mixed difference $0$, it is clear that $F^{'}_1$ and $F^{'}_2$ are $(\vec{C}_2,...,\vec{C}_2, \vec{C}_3)$-factors of $K^{*}_{4\ell+2}$. Morever, observe that each edge of pure length $d$ in $K_{4\ell+2}$ was replaced by two opposite arcs with the same endpoints; that is, arcs of pure differences $d$ and $-d$. Similarly, each edge of mixed difference $d$, where $d\neq 0$, was replaced by two opposite arcs with the same endpoints; that is, arcs of left and right mixed difference $d$.

 Finally, one edge of mixed difference $0$ gave rise to one arc of left mixed difference $0$, one arc of right difference $\infty$, and one arc of left difference $-\infty$, while the other edge of mixed difference $0$ gave rise to one arc of right mixed difference $0$, one arc of left difference $\infty$, and one arc of right difference $-\infty$.
Thus, jointly, $F_{1}^{'}$ and $F_{2}^{'}$ contain exactly one arc of each left and right pure difference in the set
$$\Big\{ 1, ..., 2\ell \Big\},$$
 exactly one arc of each left and right mixed difference in the set
$$\Big\{ 0,  1, ..., 2\ell \Big\},$$
and exactly one arc of each left and right difference $\pm\infty.$

From the properties of $F_{1}^{'}$ and $F_{2}^{'}$ we conclude that
$$\Big\{ \rho^i(F_{1}^{'}): i \in \ZZ_{2\ell+1} \Big\}\cup \Big\{ \rho^i(F_{2}^{'}): i \in \ZZ_{2\ell+1} \Big\}$$
is a  $(\vec{C}_2,...,\vec{C}_2, \vec{C}_{3})$-factorization of $K_{4\ell+1}^{*}$.
\hfill
\end{proof}


\begin{theorem}\label{thm 1,3,7}
If $n\geq 5$ and $n \equiv 1, 3,$ or $7 \pmod {8}$, then $K^{*}_{n}$ admits a $(\vec{C}_2,...,\vec{C}_2, \vec{C}_3)$-factorization.
\end{theorem}
\begin{proof}
It suffices to show that $K_{n-1}$ has 1-factors $F_1$ and $F_2$ that satisfy the assumptions of Lemma \ref{lmm:k_{4ell}} or \ref{lmm:k_{4ell+2}}. Then from Lemmas \ref{lmm:k_{4ell}} and \ref{lmm:k_{4ell+2}}, we conclude that $K^{*}_{n}$ admits a $(\vec{C}_2,...,\vec{C}_2, \vec{C}_3)$-factorization. Label the vertices of $K_{n-1}$ and for any $S\subseteq E(K_{n-1})$, define sets $\mathcal{L}(S), \mathcal{R}(S), \mathcal{M}(S), \mathcal{X}(S),$ and $\mathcal{Y}(S)$ as in Notation \ref{note1}. We have the following three cases.

{\sc Case $n \equiv 1 \pmod{8}$.} So $n=4\ell+1$ for an even integer $\ell$. We show that $K_{4\ell}$ has 1-factors $F_1$ and $F_2$ that satisfy the assumptions in Lemma \ref{lmm:k_{4ell}}. 

Consider the following subsets of the edge set of $K_{4\ell}$.
\begin{eqnarray*}
A_1&=& \{x_{i}y_{\ell-i}: i=1, 2, ..., \ell-1\},\\
A_2&=& \{x_{i}y_{\ell-1-i}: i=\ell+1, \ell+2, ..., 2\ell-2\},\quad \mbox{and}\\
A_3&=&\{x_0x_{2\ell-1}, y_{0}y_{\ell}, x_{\ell}y_{2\ell-1} \}.
\end{eqnarray*}
Let $A=A_1\cup A_2\cup A_3$. Observe that
\begin{eqnarray*}
\mathcal{X}(A_1)&=&\Big\{1, 2, ..., \ell-1 \Big\},\\
\mathcal{X}(A_2)&=&\Big\{\ell+1, ..., 2\ell-2 \Big\}, \\
\mathcal{X}(A_3)&=&\Big\{0, \ell, 2\ell-1 \Big\},\\
\mathcal{Y}(A_1)&=&\Big\{1, 2, ..., \ell-1 \Big\},\\
\mathcal{Y}(A_2)&=&\Big\{\ell+1, ..., 2\ell-2 \Big\}, \\
\mathcal{Y}(A_3)&=&\Big\{0, \ell, 2\ell-1 \Big\}.
\end{eqnarray*}
 Hence, $\mathcal{X}(A)=\mathcal{Y}(A)=\ZZ_{2\ell}$ and $A$ covers each vertex of $K_{4\ell}$ exactly once. Thus, $A$ is a perfect matching, and so the subgraph $F_1=(V, A)$ is a 1-factor of $K_{4\ell}$.

Observe also that the sets $\mathcal{R}(A_1), \mathcal{R}(A_2), \mathcal{L}(A_1), $ and $\mathcal{L}(A_2)$ are all empty, while 
\begin{eqnarray*}
\mathcal{M}(A_1)&=&\Big\{ 0, 2, 4, ..., \ell-2, \ell+2, \ell+4, ..., 2\ell-2\Big\},\\
\mathcal{M}(A_2)&=&\Big\{1, 3, ..., \ell-3, \ell+3, \ell+5, ..., 2\ell-1 \Big\}.\\
\end{eqnarray*}
And also for the set $A_3$ we have 
\begin{eqnarray*}
\mathcal{M}(A_3)&=&\Big\{ \ell-1\Big\},\\
\mathcal{R}(A_3)&=&\Big\{\ell \Big\},\\
\mathcal{L}(A_3)&=&\Big\{ 1\Big\}.
\end{eqnarray*}
Therefore, $F_1$ is a 1-factor of $K_{4\ell}$ that contains exactly one edge of each left pure length in  
$$\mathcal{L}(A)= \Big\{1 \Big\},$$
exactly one edge of each right pure length in
$$\mathcal{R}(A)= \Big\{\ell \Big\},$$
and exactly one edge of each mixed difference in 
$$\mathcal{M}(A)= \Big\{ 0, 1, 2, ..., \ell-1, \ell+2, \ell+3, ..., 2\ell-1\Big\}.$$
 
Next, we construct the second 1-factor.
Consider the following subsets of the edge set of $K_{4\ell}$.
\begin{eqnarray*}
B_1&=&\Big\{x_{i}x_{2\ell-i}:i= 1, 2, ...,\frac{\ell}{2}-1\Big\},\\ 
B_2&= &\Big\{x_{i}x_{2\ell-i-1}: i=\frac{\ell}{2}, \frac{\ell}{2}+1, ...,\ell-2\Big\}, \\
B_3&=&\Big\{y_{i}y_{\ell-i}: i=1, 2, ...,\frac{\ell}{2}-1\Big\},\\
B_4&=& \Big\{y_{i}y_{\ell-1-i}: i=\ell, \ell+1, ..., \frac{3\ell}{2}-1\Big\},\\
B_5&=& \Big\{x_{0}x_{\ell}, x_{\ell-1}y_{0}, x_{\frac{3\ell}{2}}y_{\frac{\ell}{2}}\Big\}.
\end{eqnarray*}
 Let $B=B_1\cup B_2\cup B_3\cup B_4\cup B_5$. Observe that the sets $\mathcal{X}(B_3),\mathcal{X}(B_4), \mathcal{Y}(B_1),$ and $\mathcal{Y}(B_2)$ are all empty, whereas
\begin{eqnarray*}
\mathcal{X}(B_1)&=&\Big\{ 1, 2, ..., \frac{\ell}{2}-1, \frac{3\ell}{2}+1, ..., 2\ell-1\Big\},\\
\mathcal{X}(B_2)&=&\Big\{\frac{\ell}{2}, \frac{\ell}{2}+1, ..., \ell-2, \ell+1,..., \frac{3\ell}{2}-1 \Big\},\\
\mathcal{Y}(B_3)&=&\Big\{ 1, 2, ..., \frac{\ell}{2}-1,\frac{\ell}{2}+1, ..., \ell-1 \Big\},\\
\mathcal{Y}(B_4)&=&\Big\{ \ell, \ell+1, ..., 2\ell-1\Big\},\\
 \mathcal{X}(B_5)&=&\Big\{ 0, \ell-1, \ell, \frac{3\ell}{2}\Big\},\\
 \mathcal{Y}(B_5)&=&\Big\{0, \frac{\ell}{2} \Big\}.
 \end{eqnarray*}
Hence, $ \mathcal{X}(B)= \mathcal{Y}(B)=\ZZ_{2\ell}$, and $B$ covers each vertex of $K_{4\ell}$ exactly once. Thus, $B$ is perfect matching, and so the subgraph $F_2=(V, B)$ is a 1-factor of $K_{4\ell}$.

Also observe that we have 
\begin{eqnarray*}
\mathcal{L}(B_1)&=&\Big\{2, 4 ..., \ell-2 \Big\},\\
\mathcal{L}(B_2)&=&\Big\{ 3, 5 ..., \ell-1 \Big\},\\
\mathcal{R}(B_3)&=&\Big\{ 2, 4 ..., \ell-2 \Big\},\\
\mathcal{R}(B_4)&=&\Big\{ 1,3 ..., \ell-1 \Big\},\\
\mathcal{L}(B_5)&=&\Big\{\ell \Big\},\\
\mathcal{M}(B_5)&=&\Big\{ \ell, \ell+1\Big\},
\end{eqnarray*}
whereas $\mathcal{R}(B_1),\mathcal{R}(B_2),\mathcal{R}(B_5),\mathcal{L}(B_3),\mathcal{L}(B_4),\mathcal{M}(B_1), \mathcal{M}(B_2),\mathcal{M}(B_3),$ and $ \mathcal{M}(B_4)$ are all empty.

Therefore, $F_2$ is a 1-factor of $K_{4\ell}$ that contains exactly one edge of each left pure length in  
$$\mathcal{L}(B)= \Big\{2, 3, ..., \ell \Big\},$$
exactly one edge of each right pure length in
$$\mathcal{R}(B)= \Big\{1, 2, ...,\ell-1 \Big\},$$
and exactly one edge of each mixed difference in 
$$\mathcal{M}(B)= \Big\{ \ell, \ell+1\Big\}.$$
It follows that $F_1$ and $F_2$ jointly contain exactly one edge of each left and right pure length in 
$$\Big\{1, 2,..., \ell \Big\}$$
and exactly one edge of each mixed difference in
$$\Big\{0,1, 2,..., 2\ell-1 \Big\}.$$
 Observe also that $F_1$ and $F_2$ each contain exactly one edge of pure length $\ell$.
Thus, $F_{1}$ and $F_{2}$ are 1-factors of $K_{4\ell}$ that satisfy the assumptions in Lemma \ref{lmm:k_{4ell}}. 
For an example, the 1-factors $F_1$ and $F_2$ of $K_{24}$ are illustrated  in Figure \ref{fff5}.

\begin{figure}[h!]
\begin{center}
\centerline{\includegraphics[scale=.68]{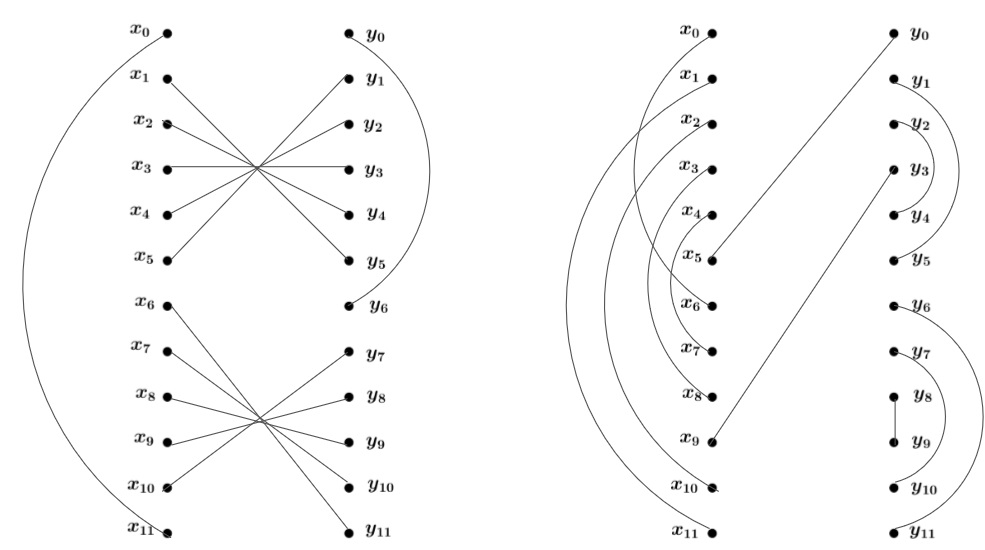}}
\caption{1-factors $F_1$ and $F_2$ of $K_{24}$}.
\label{fff5}
\end{center}
\end{figure}

{\sc Case $n \equiv 3 \pmod{8}$.} So $n=4\ell+3$ for an even integer $\ell$. In this case, we show $K_{4\ell+2}$ has two 1-factors that satisfy the assumptions in Lemma \ref{lmm:k_{4ell+2}}.

Consider the following subsets of the edge set of $K_{4\ell+2}$. 
\begin{eqnarray*}
A_1 &= &\Big\{x_{i}y_{\ell-i}: i=1, 2, ...,\frac{\ell}{2}-1,\frac{\ell}{2}+1, ..., \ell-1\Big\},\\
A_2 &= & \Big\{x_{i}y_{3\ell+1-i}: i=\ell+1, \ell+2, ..., 2\ell\Big\}, \quad \mbox{and} \\
A_3 &= &\Big\{x_0y_0, x_{\frac{\ell}{2}}x_{\ell}, y_{\frac{\ell}{2}}y_{\ell}\Big\}.
\end{eqnarray*}
Let $A=A_1\cup A_2 \cup A_3$. Observe that 
\begin{eqnarray*}
\mathcal{X}(A_1)&=&\Big\{ 1, 2, ...\frac{\ell}{2}-1,\frac{\ell}{2}+1,..., \ell-1 \Big\},\\
\mathcal{X}(A_2)&=&\Big\{ \ell+1,\ell+2, ..., 2\ell \Big\},\\
\mathcal{X}(A_3)&=&\Big\{0, \frac{\ell}{2}, \ell \Big\},\\
\mathcal{Y}(A_1)&=&\Big\{1, 2, ...,\frac{\ell}{2}-1,\frac{\ell}{2}+1,..., \ell-1\Big\},\\
\mathcal{Y}(A_2)&=&\Big\{ \ell+1,\ell+2, ..., 2\ell \Big\}, \mbox{and}\\
\mathcal{Y}(A_3)&=&\Big\{0, \frac{\ell}{2}, \ell \Big\}.
\end{eqnarray*}
 Thus, $\mathcal{X}(A)=\mathcal{Y}(A)=\ZZ_{2\ell+1}$.
Hence, observe that $A$ covers exactly once each vertex $x_i$ and $y_i$ of $K_{4\ell+2}$. Therefore, the subgraph $F_1=(V, A)$ is a 1-factor of $K_{4\ell+2}$.
 
Also observe that for the sets $A_1$, $A_2$, and $A_3$ we have 
\begin{eqnarray*}
\mathcal{M}(A_1)&=&\Big\{2, 4, ..., \ell-4, \ell-2, \ell+3, \ell+5, ..., 2\ell-3, 2\ell-1\Big\},\\
\mathcal{M}(A_2)&=&\Big\{1, 3, ..., \ell-3,\ell-1, \ell+2,\ell+4,..., 2\ell-2, 2\ell\Big\},\\
\mathcal{M}(A_3)&=&\Big\{0\Big\},\\
\mathcal{L}(A_3)&=&\Big\{\frac{\ell}{2}\Big\},\\
\mathcal{R}(A_3)&= &\Big\{\frac{\ell}{2}\Big\},
\end{eqnarray*}
whereas $\mathcal{L}(A_1), \mathcal{R}(A_1),\mathcal{L}(A_2),$ and $ \mathcal{R}(A_2)$ are all empty. Therefore,
\begin{align*}\label{EQS 1}
\mathcal{L}(A)&=\Big\{\frac{\ell}{2}\Big\},\\
\mathcal{R}(A)&=\Big\{\frac{\ell}{2}\Big\},\tag{1}\\
\mathcal{M}(A)&= \ZZ_{2\ell+1}-\Big\{\ell, \ell+1\Big\}.
\end{align*}
Next, we construct the second 1-factor. First, we define

\begin{equation*}
j=\left\{ \begin{array}{ll}
\frac{\ell-10}{4} & \quad\quad \mbox{if $\ell \equiv 2 \pmod{4}$};\\
  \frac{3\ell-8}{4} & \quad\quad \mbox{if $\ell \equiv 0 \pmod{4}$}.\end{array} \right.
\end{equation*}

 Consider the following subsets of the edge set of $K_{4\ell+2}$. 
\begin{eqnarray*}
B_1&=&\Big\{y_{i}y_{2\ell-2-i}: i=0, 1, ..., j \Big\},\\
B_2&=&\Big\{y_{i}y_{2\ell-i}: i= j+3, j+4, ..., \ell-1\Big\} ,\\
B_3&=&\Big\{x_{i}x_{2\ell-2-i}: i=0, 1, ..., j\Big\},\\
B_4&=&\Big\{x_{i}x_{2\ell-i}: i= j+3, j+4, ..., \ell-1\Big\}, \quad\mbox{and}\\
B_5&= &\Big\{x_{j+1}x_{j+2}, y_{j+1}y_{j+2}, x_{2\ell-1}y_{2\ell-1}, x_{\ell}y_{2\ell}, x_{2\ell}y_{\ell}\Big\}.
\end{eqnarray*}
{\sc Subcase $\ell\geq 8$.} Let $B=B_1\cup B_2\cup B_3 \cup B_4 \cup B_5$. Observe that $\mathcal{X}(B_1), \mathcal{X}(B_2), \mathcal{Y}(B_3)$, and $\mathcal{Y}(B_4)$ are all empty, while 
\begin{eqnarray*}
 \mathcal{Y}(B_1)&=&\Big\{0, 1, ..., j, 2\ell-j-2, 2\ell-j-1, ...,2\ell-2 \Big\},\\
 \mathcal{Y}(B_2)&=&\Big\{ j+3, j+4, ..., \ell-1, \ell+1, ...,2\ell-j-3 \Big\},\\
 \mathcal{X}(B_3)&=&\Big\{0, 1, ..., j, 2\ell-j-2, 2\ell-j-1, ...,2\ell-2 \Big\},\\
\mathcal{X}(B_4)&=&\Big\{ j+3, j+4, ..., \ell-1, \ell+1, ...,2\ell-j-3 \Big\},\\
\mathcal{X}(B_5)&=&\Big\{ j+1, j+2, \ell, 2\ell-1, 2\ell\Big\},\\
\mathcal{Y}(B_5)&=&\Big\{ j+1, j+2, \ell, 2\ell-1, 2\ell\Big\}.
\end{eqnarray*}
 Since $\ell\geq 8$, we have $0\leq j \leq \ell-4$. Consequently, $\mathcal{X}(B)=\mathcal{Y}(B)= \ZZ_{2\ell+1}$, that is, $B$ covers exactly once each vertex $x_i$ and $y_i$ of $K_{4\ell+2}$. Thus, $B$ is a perfect maching, and the subgraph $F_2=(V, B)$ is a 1-factor of $K_{4\ell+2}$. 

Furthermore, observe that if $\ell \equiv 2 \pmod{4}$, then 
\begin{eqnarray*}
\mathcal{R}(B_1)&=&\Big\{3, 5, ..., \frac{\ell}{2}-2\Big\},\\
\mathcal{R}(B_2)&=&\Big\{2, 4, ..., \ell\Big\}\cup \Big\{\frac{\ell}{2}+2, \frac{\ell}{2}+4, ..., \ell-1 \Big\},\\
\mathcal{L}(B_3)&=&\Big\{3, 5, ..., \frac{\ell}{2}-2\Big\},\\
\mathcal{L}(B_4)&=&\Big\{2, 4, ..., \ell\Big\}\cup \Big\{\frac{\ell}{2}+2, \frac{\ell}{2}+4, ..., \ell-1 \Big\},
\end{eqnarray*}
whereas $\mathcal{L}(B_1), \mathcal{L}(B_2), \mathcal{R}(B_3), \mathcal{R}(B_4),\mathcal{M}(B_1), \mathcal{M}(B_2), \mathcal{M}(B_3),$ and $ \mathcal{M}(B_4) $ are all empty.

And if $\ell \equiv 0 \pmod{4}$, then 
\begin{eqnarray*}
 \mathcal{R}(B_1)&=&\Big\{3, 5, ..., \ell-1\Big\}\cup \Big\{\frac{\ell}{2}+2, \frac{\ell}{2}+4, ..., \ell \Big\},\\
 \mathcal{R}(B_2)&=&\Big\{2, 4, ..., \frac{\ell}{2}-2\Big\},\\
 \mathcal{L}(B_3)&=&\Big\{3, 5, ..., \ell-1\Big\}\cup \Big\{\frac{\ell}{2}+2, \frac{\ell}{2}+4, ..., \ell \Big\},\\
\mathcal{L}(B_4)&=&\Big\{2, 4, ..., \frac{\ell}{2}-2\Big\},
\end{eqnarray*}
whereas $\mathcal{L}(B_1),\mathcal{L}(B_2), \mathcal{R}(B_3), \mathcal{R}(B_4),\mathcal{M}(B_1), \mathcal{M}(B_2),
\mathcal{M}(B_3),$ and $\mathcal{M}(B_4)$ are all empty.

Also observe that in both cases
\begin{eqnarray*}
\mathcal{R}(B_5)=\mathcal{L}(B_5) =\Big\{1\Big\},\\
\mathcal{M}(B_5)=\Big\{0, \ell, \ell+1\Big\}.
\end{eqnarray*}
Thus, $F_2=(V, B)$ is a 1-factor of $K_{4\ell+2}$ with
\begin{align*}\label{EQS 2}
\mathcal{L}(B) &=L- \Big\{\frac{\ell}{2}\Big\},\\
\mathcal{R}(B) &=L-\Big\{\frac{\ell}{2}\Big\},\tag{2}\\
\mathcal{M}(B) &=L-\Big\{ 0, \ell, \ell+1\Big\}.
\end{align*}

Next, suppose that $\ell<8$. Note that in this case $\ell\in \{2, 4, 6\}$.

{\sc Subcase $\ell=6$}. Let $B=B_2\cup B_4\cup B_5$. Observe that we have $\mathcal{X}(B_2)=\mathcal{Y}(B_4)=\emptyset$, whereas 
\begin{eqnarray*}
\mathcal{X}(B_4)=&\mathcal{Y}(B_2)&=\Big\{ 2, 3, 4, 5, 7, 8, 9, 10\Big\},\\
\mathcal{X}(B_5)=&\mathcal{Y}(B_5)&=\Big\{ 0, 1, 6, 11, 12\Big\}.
\end{eqnarray*}
Hence, $\mathcal{X}(B)=\mathcal{Y}(B)=\ZZ_{13}$, and $F_2=(V, B)$ is a 1-factor of $K_{26}$.

Furthermore, observe that $\mathcal{L}(B_2), \mathcal{R}(B_4),\mathcal{M}(B_2)$ and $\mathcal{M}(B_4)$ are all empty, whereas
\begin{eqnarray*}
\mathcal{R}(B_2)=\mathcal{L}(B_4)&=&\Big\{2, 4, 5, 6\Big\},\\
\mathcal{R}(B_5)=\mathcal{L}(B_5) &=&\Big\{1\Big\},\quad \mbox{and}\\
\mathcal{M}(B_5)&=&\Big\{0, 6, 7\Big\}.
\end{eqnarray*}
Cosequently, \eqref{EQS 2} holds in this case as well.

{\sc Subcase $\ell=4$}. Let $B=B_1\cup B_3\cup B_5$. We have 
\begin{eqnarray*}
\mathcal{Y}(B_1)=\mathcal{X}(B_3)=\Big\{0, 1, 5, 6 \Big\},\\
\mathcal{Y}(B_5)=\mathcal{X}(B_5)=\Big\{2, 3, 4, 7, 8 \Big\},
\end{eqnarray*}
whereas $\mathcal{Y}(B_3)=\mathcal{X}(B_1)=\emptyset$.

Clearly, $\mathcal{X}(B)=\mathcal{Y}(B)=\ZZ_{9}$. Thus, $F_2=(V, B)$ is a 1-factor of $K_{18}$.

Also, observe that $\mathcal{R}(B_3),\mathcal{L}(B_1),\mathcal{M}(B_1)$ and $\mathcal{M}(B_3) $ are all empty, while
\begin{eqnarray*}
\mathcal{R}(B_1)=\mathcal{L}(B_3)=\Big\{3,4\Big\},\\
\mathcal{R}(B_5)=\mathcal{L}(B_5) =\Big\{1\Big\},\\
\mathcal{M}(B_5)=\Big\{0, 4, 5\Big\}.
\end{eqnarray*}
Consequently, \eqref{EQS 2} holds in this case, too.

{\sc Subcase $\ell=2$.} In this case let 
$$B= \Big\{ x_0x_2, y_0y_2, x_1y_3, x_3y_1, x_4y_4  \Big\}.$$
It is easy to verify that $F_2=(V, B)$ is a 1-factor of $K_{10}$ and satisfies \eqref{EQS 2}.

From \eqref{EQS 1} and \eqref{EQS 2} conclude that in all cases $F_1$ and $ F_2$ jointly contain exactly one edge of each left and right pure length in $L$, and exactly one edge of each mixed difference in $\ZZ_{2\ell+1}^{*}$.
Observe also that each of $F_1$ and $F_2$ contains exactly one edge of mixed difference $0$.
Thus, $F_{1}$ and $F_{2}$ are 1-factors of $K_{4\ell+2}$ that satisfy the assumptions in Lemma \ref{lmm:k_{4ell+2}}.
For an example, the 1-factors $F_1$ and $F_2$ of $K_{34}$ are illustrated  in Figure \ref{fff34}.

\begin{figure}[h!]
\begin{center}
\centerline{\includegraphics[scale=.7]{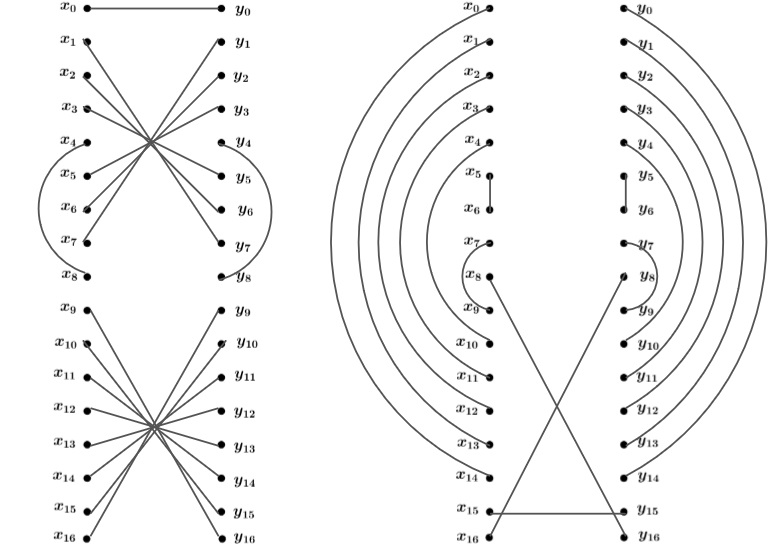}}
\caption{1-factors $F_1$ and $F_2$ of $K_{34}$}.
\label{fff34}
\end{center}
\end{figure}

{\sc Case $n \equiv 7 \pmod{8}$.} So $n=4\ell+3$ for an odd integer $\ell$. Again, we shall construct 1-factors $F_1$ and $F_2$ of $K_{4\ell+2}$ that satisfy the assumptions in Lemma \ref{lmm:k_{4ell+2}}.

First, assume that $\ell>1$. Define the subsets of the edge set of $K_{4\ell+2}$ as follows.
\begin{eqnarray*}
A_1&=&\Big\{x_iy_{\ell-i}: i=1, 2,..., \ell-1 \Big\},\\
A_2&=&\Big\{x_iy_{\ell-i}: i=\ell+1, \ell+2, ..., \frac{3\ell-1}{2}, \frac{3\ell+3}{2}, \frac{3\ell+5}{2}...,2\ell \Big\},\\
A_3&=&\Big\{x_0y_0, x_{\ell}x_{\frac{3\ell+1}{2}}, y_{\ell}y_{\frac{3\ell+1}{2}} \Big\}.
\end{eqnarray*}
 Let $A=A_1\cup A_2 \cup A_3$. Observe that
\begin{eqnarray*}
\mathcal{X}(A_1)=\mathcal{Y}(A_1)&=&\Big\{ 1, 2, ..., \ell-1\Big\},\\
\mathcal{X}(A_2)=\mathcal{Y}(A_2)&=&\Big\{ \ell+1, \ell+2, ...,\frac{3\ell-1}{2}, \frac{3\ell+3}{2}, \frac{3\ell+5}{2}...,2\ell \Big\},\\
\mathcal{X}(A_3)=\mathcal{Y}(A_3)&=&\Big\{ 0, \ell, \frac{3\ell+1}{2}\Big\}.
\end{eqnarray*}
Hence, $\mathcal{X}(A)=\mathcal{Y}(A)=\ZZ_{2\ell+1}$, and $A$ covers exactly once each vertex of $K_{4\ell+2}$. Therefore, $A$ is a perfect maching and $F_1=(V, A)$ is a 1-factor of $K_{4\ell+2}$.

Also observe that $\mathcal{L}(A_1), \mathcal{L}(A_2), \mathcal{R}(A_1), $ and $\mathcal{R}(A_2)$ are all empty, whereas 
\begin{eqnarray*}
\mathcal{M}(A_1)&=&\Big\{1, 3, ...,\ell-2, \ell+3, \ell+5, ..., 2\ell \Big\},\\
\mathcal{M}(A_2)&=&\Big\{2, 4, ...,\ell-1, \ell+2, \ell+4, ..., 2\ell -1\Big\},\\
\mathcal{M}(A_3)&=&\Big\{0\Big\},\\
\mathcal{L}(A_3)&=&\Big\{\frac{\ell+1}{2}\Big\},\\
\mathcal{R}(A_3)&=&\Big\{\frac{\ell+1}{2}\Big\}.
\end{eqnarray*}
Thus, $F_1=(V, A)$ is a 1-factor of $K_{4\ell+2}$ with the following properties
\begin{align*}\label{EQS3}
\mathcal{L}(A)&=\Big\{\frac{\ell+1}{2}\Big\},\\
\mathcal{R}(A)&=\Big\{\frac{\ell+1}{2}\Big\},\tag{3}\\
\mathcal{M}(A)&=\ZZ_{2\ell+1}-\Big\{\ell, \ell+1\Big\}.
\end{align*}
Now, we construct the second 1- factor. For this purpose, we consider two following subcases.

{\sc Subcase $\ell \equiv 1 \pmod{4}$.} 

Define the following subsets of the edge set of $K_{4\ell+2}$.
\begin{eqnarray*}
B_1&=&\Big\{ x_ix_{2\ell-2-i}: i=0, 1, ..., \frac{\ell-9}{4}\Big\},\\
B_2&=&\Big\{x_ix_{2\ell-i}: i=\frac{\ell+3}{4}, \frac{\ell+7}{4}..., \ell-1 \Big\},\\
B_3&=&\Big\{y_iy_{2\ell-2-i}: i=0, 1, ..., \frac{\ell-9}{4} \Big\},\\
B_4&=&\Big\{y_iy_{2\ell-i}: i=\frac{\ell+3}{4}, \frac{\ell+7}{4}..., \ell-1 \Big\},\mbox{\quad and}\\
B_5&=&\Big\{ x_{\frac{\ell-5}{4}} x_{\frac{\ell-1}{4}}, y_{\frac{\ell-5}{4}} y_{\frac{\ell-1}{4}}, x_{\ell}y_{2\ell}, x_{2\ell-1}y_{2\ell-1}, x_{2\ell}y_{\ell} \Big\}.
\end{eqnarray*} 
First, assume $\ell\geq 9$. Let $B=B_1\cup B_2\cup B_3\cup B_4 \cup B_5$. Observe that while $ \mathcal{X}(B_3),  \mathcal{X}(B_4),  \mathcal{Y}(B_1),$ and $ \mathcal{Y}(B_2)$ are all empty, we have
 \begin{eqnarray*}
 \mathcal{X}(B_1)= \mathcal{Y}(B_3)&=&\Big\{ 0, 1, ..., \frac{\ell-9}{4}, \frac{7\ell+1}{4}, \frac{7\ell+5}{4} ..., 2\ell-2 \Big\},\\
  \mathcal{X}(B_2)= \mathcal{Y}(B_4)&=&\Big\{ \frac{\ell+3}{4}, \frac{\ell+7}{4} ...,\ell-1, \ell+1, \ell+2...,\frac{7\ell-3}{4} \Big\},\\
 \mathcal{X}(B_5)= \mathcal{Y}(B_5)&=& \Big\{\frac{\ell-5}{4}, \frac{\ell-1}{4}, \ell, 2\ell-1, 2\ell \Big\}.
 \end{eqnarray*}
Hence, $\mathcal{X}(B)=\mathcal{Y}(B) =\ZZ_{2\ell+1}$, and $B$ covers each vertex of $K_{4\ell+2}$ exactly once.
Therefore, $F_2=(V, B)$ is a 1- factor of $K_{4\ell+2}$.

Observe also that while $\mathcal{R}(B_1), \mathcal{R}(B_2), \mathcal{L}(B_3), \mathcal{L}(B_4),\mathcal{M}(B_1),\mathcal{M}(B_2), \mathcal{M}(B_3)$ and $\mathcal{M}(B_4)$ are all empty, we have 
\begin{eqnarray*}
\mathcal{L}(B_1)=\mathcal{R}(B_3)&= &\Big \{ 3, 5, ..., \frac{\ell-3}{2}\Big\},\\
\mathcal{L}(B_2)=\mathcal{R}(B_4)&=&\Big \{ 2, 4, ..., \ell-1\Big\} \bigcup \Big\{\frac{\ell+5}{2},  \frac{\ell+9}{2}, ..., \ell \Big \},\\
\mathcal{L}(B_5)= \mathcal{R}(B_5)&=&\Big\{ 1\Big\}, \mbox{\quad and}\\
\mathcal{M}(B_5)&=&\Big\{ 0, \ell, \ell+1\Big\}.
\end{eqnarray*}
Therefore, $F_2=(V, B)$ is a 1-factor of $K_{4\ell+2}$ with 
\begin{align*}\label{EQA 4}
\mathcal{L}(B)&=L-\Big \{\frac{\ell+1}{2} \Big\},\\
\mathcal{R}(B)&=L-\Big \{\frac{\ell+1}{2} \Big\},\tag{4}\\
\mathcal{M}(B)&= \Big\{ 0, \ell, \ell+1\Big\}.
\end{align*}

Next, suppose that $\ell< 9$. In this case we have $\ell= 5$. Let $B= B_2\cup B_4 \cup B_5$. 

Observe that while  $\mathcal{Y}(B_2)$ and $\mathcal{X}(B_4)$ are empty, we have 
\begin{eqnarray*}
\mathcal{X}(B_2)= \mathcal{Y}(B_4)=\Big \{ 2, 3, 4, 6, 7, 8\Big\},\\
\mathcal{X}(B_5)= \mathcal{Y}(B_5)=\Big \{0, 1, 5, 9, 10\Big\}.
\end{eqnarray*}
Clearly, $F_2=(V, B)$ is a 1-factor of $K_{22}$. 

Also observe that 
\begin{eqnarray*}
\mathcal{L}(B_2)&= &\mathcal{R}(B_4)=\Big \{ 2, 4, 5\Big\},\\
\mathcal{L}(B_5)&=& \mathcal{R}(B_5)=\Big \{ 1\Big\},\\
\mathcal{M}(B_5)&=&\Big \{0, 5, 6\Big\},
\end{eqnarray*} 
while $\mathcal{R}(B_2), \mathcal{L}(B_4), \mathcal{M}(B_2), $ and $\mathcal{M}(B_4)$ are all empty. Hence, \eqref{EQA 4} holds in this case as well.

{\sc Subcase $\ell \equiv 3 \pmod{4}$.} 

First, assume $\ell\geq 7$. Consider the following subsets of the edge set of $K_{4\ell+2}$.
\begin{eqnarray*}
B_1&=&\Big\{x_ix_{2\ell-i}: i= 0, 1, ..., \frac{\ell-3}{4}, \frac{3\ell+3}{4}, \frac{3\ell+7}{4} ..., \ell-1 \Big\},\\
B_2&=&\Big\{ x_ix_{2\ell-i-1}: i=\frac{\ell+1}{4}, \frac{\ell+5}{4}..., \frac{3\ell-5}{4} \Big\},\\
B_3&=&\Big\{y_iy_{2\ell-i}: i= 0, 1, ..., \frac{\ell-3}{4}, \frac{3\ell+3}{4},\frac{3\ell+7}{4} ..., \ell-1 \Big\},\\
B_4&=&\Big\{y_iy_{2\ell-i-1}: i=\frac{\ell+1}{4},\frac{\ell+5}{4}..., \frac{3\ell-5}{4} \Big\},\\
B_5&=&\Big\{x_{\frac{3\ell-1}{4}}y_{\frac{7\ell-1}{4}}, x_{\ell}y_{\ell}, x_{\frac{7\ell-1}{4}}y_{\frac{3\ell-1}{4}} \Big\}.
\end{eqnarray*}
 Let $B= B_1\cup B_2\cup B_3\cup B_4\cup B_5$. Observe that $\mathcal{Y}(B_1), \mathcal{Y}(B_2), \mathcal{X}(B_3),$ and $\mathcal{X}(B_4)$ are all empty, whereas
\begin{eqnarray*}
\mathcal{X}(B_1)=\mathcal{Y}(B_3)&=& \Big\{0, 1, ..., \frac{\ell-3}{4}, \frac{3\ell+3}{4}, \frac{3\ell+7}{4} ..., \ell-1\Big\}\cup\\ && \Big\{ \ell+1,\ell+2, ..., \frac{5\ell-3}{4},\frac{7\ell+3}{4}, \frac{7\ell+7}{4} ..., 2\ell \Big\},\\
\vspace*{10mm}
\mathcal{X}(B_2)=\mathcal{Y}(B_4)&=&\Big\{\frac{\ell+1}{4},\frac{\ell+5}{4}..., \frac{3\ell-5}{4},  \frac{5\ell+1}{4}, \frac{5\ell+5}{4}..., \frac{7\ell-5}{4}\Big\},\\
\mathcal{X}(B_5)=\mathcal{Y}(B_5)&=&\Big\{\frac{3\ell-1}{4}, \ell, \frac{7\ell-1}{4} \Big\}.
\end{eqnarray*}
Thus, $\mathcal{X}(B)=\mathcal{Y}(B)=\ZZ_{2\ell+1}$, and $B $ covers each vertex of $K_{4\ell+2}$ exactly once. Therefore, $F_2=(V, B)$ is a 1-factor of $K_{4\ell+2}$.

Also observe that 
\begin{eqnarray*}
\mathcal{L}(B_1)&= & \mathcal{R}(B_3)= \Big\{1, 3, ..., \frac{\ell-1}{2}\Big\} \cup \Big\{2, 4, ..., \frac{\ell-3}{2} \Big\},\\
\mathcal{L}(B_2)&= & \mathcal{R}(B_4)=  \Big\{ \frac{\ell+5}{2}, \frac{\ell+9}{2}, ..., \ell-1\Big\} \cup \Big\{\frac{\ell+3}{2}, \frac{\ell+7}{2}, ..., \ell \Big\},\\
\mathcal{M}(B_5)&= & \Big\{ 0, \ell, \ell+1 \Big\},
\end{eqnarray*}
 while $\mathcal{R}(B_1), \mathcal{R}(B_2), \mathcal{L}(B_3), \mathcal{L}(B_4), \mathcal{L}(B_5), $ and $\mathcal{R}(B_5)$ are all empty.
 
 Hence, $F_2=(V, B)$ is a 1-factor of $K_{4\ell+2}$ satisfying \eqref{EQA 4}.
 
For $\ell< 7$, we have $\ell=3$. Let 
$$ B=\Big\{ x_0x_6, x_1 x_4, y_0 y_6, y_1y_4, x_2y_5, x_3y_3, x_5y_2  \Big\}.$$
It can be verified easily that $F_2=(V, B)$ is a 1-factor of $K_{15}$ that satisfies \eqref{EQA 4}. Hence in all cases with $\ell>1$, we have constructed 1-factors $F_1$ and $F_2$ that satisfy properties \eqref{EQS3} and \eqref{EQA 4}.

Finally, suppose that $\ell=1$. In this case let
\begin{eqnarray*}
A=\Big\{x_0y_0, x_1x_2, y_1y_2 \Big\},\\
B=\Big\{ x_0y_2, x_1y_1, x_2y_0\Big\}.\\
\end{eqnarray*}
It is easy to see that $F_1=(V, A)$ and $F_2=(V, B)$ are two 1-factors of $K_{7}$ with
\begin{eqnarray*}
\mathcal{L}(A)&=&\mathcal{R}(A)= \Big\{1 \Big\},\\
\mathcal{M}(A)&=&\Big\{0 \Big\},\\
\mathcal{M}(B)&=& \Big\{0, 1, 2 \Big\},
\end{eqnarray*}
while $\mathcal{L}(B)$ and $\mathcal{R}(B)$ are empty.

Consequently, in all cases $ F_1$ and $F_2$ jointly contain exactly one edge of each pure length in $L$, and exactly one edge of each mixed difference in $\ZZ^{*}_{2\ell+1}$, and also each of $ F_1$ and $F_2$ contains exactly one edge of mixed difference $0$. Thus, $F_1$ and $F_2$ are two 1-factors of $K_{4\ell+2}$ that satisfy the assumptions in Lemma \ref{lmm:k_{4ell+2}}.
For an example, the 1-factors $F_1$ and $F_2$ of $K_{30}$ are illustrated  in Figure \ref{fff30}.

\begin{figure}[h!]
\begin{center}
\centerline{\includegraphics[scale=.7]{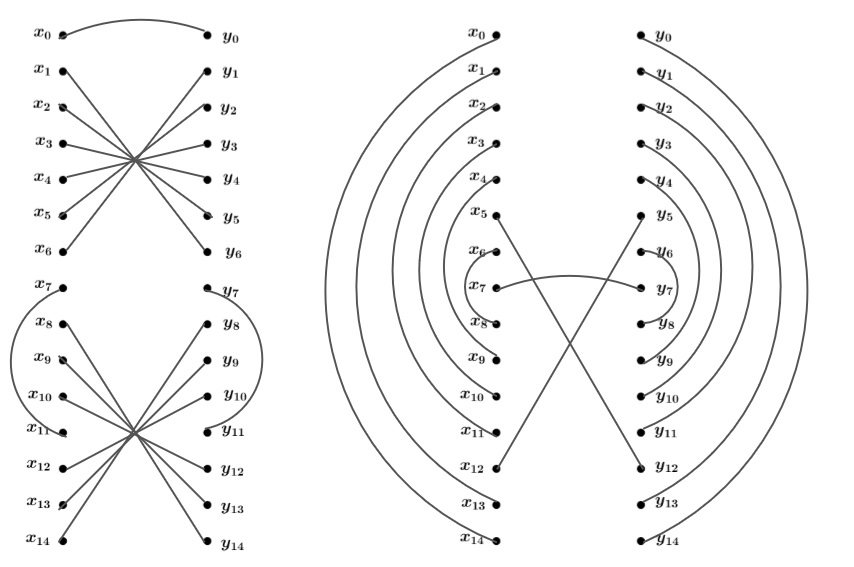}}
\caption{1-factors $F_1$ and $F_2$ of $K_{30}$}.
\label{fff30}
\end{center}
\end{figure}
\end{proof}
 
Observe that we have one missing case in Theorem \ref{thm 1,3,7}, namely, $ n \equiv 5 \pmod {8}$. In this case we have the following lemma, which shows that the construction used in the proof of Theorem \ref{thm 1,3,7} cannot be used in this case. However, we do not know whether or not $K_{n}^{*}$ admits a $(\vec{C}_2, ..., \vec{C}_2, \vec{C}_3)$-factorization in this case.


\begin{lem}\label{lmm:ell is odd }
If $n \equiv 5 \pmod {8}$, then $K_{n-1}$ does not have 1-factors $F_1$ and $F_2$ that satisfy the assumptions in Lemma \ref{lmm:k_{4ell}}.

\end{lem}
\begin{proof}
Let $ n \equiv 5 \pmod {8}$, so $n=4\ell+1$, where $\ell$ is an odd integer. Suppose, to the contrary, that there exist 1-factors $F_1$ and $F_2$ of $K_{4\ell}$ that satisfy the assumptions in Lemma \ref{lmm:k_{4ell}}. 

For $i\in\{1,2\}$, define the following parameters:
\begin{eqnarray*}
\varepsilon^{o}_{i}&=&\Big|\Big\{x_j: x_jy_{j+d}\in E(F_i), ~ j, d\in \ZZ_{2\ell}, ~j \equiv1\pmod{2},~ d\equiv0\pmod {2} \Big\}\Big|,\\
\varepsilon^{e}_{i}&=&\Big|\Big\{x_j: x_jy_{j+d}\in E(F_i), ~ j, d\in \ZZ_{2\ell},~j \equiv 0 \pmod {2},~ d\equiv 0\pmod {2} \Big\}\Big|,\\
\omega^{o}_{i}&=&\Big|\Big\{x_j: x_jy_{j+d}\in E(F_i), ~ j, d\in \ZZ_{2\ell}, ~j \equiv 1 \pmod {2},~ d\equiv 1\pmod {2} \Big\}\Big|,\\
\lambda_{i}&=&\Big|\Big\{x_jx_{j+d}\in E(F_i):j\in \ZZ_{2\ell},~ d\in L,~d\equiv 1\pmod {2}  \Big\}\Big|,\\
\rho_{i}&=&\Big|\Big\{ y_jy_{j+d}\in E(F_i): j\in \ZZ_{2\ell},~ d\in L, ~d\equiv 1\pmod {2}\Big\}\Big|.
\end{eqnarray*}
Observe that since $\ell \equiv 1\pmod{2}$, the total number of edges of odd left and odd right pure length in $K_{4\ell}$ is 
\begin{align*}\label{EQS 5}
\lambda_{1}+\lambda_{2}&=\frac{\ell+1}{2}~\mbox{and}\\
\rho_{1}+\rho_{2}&=\frac{\ell+1}{2}, ~\mbox{respectively}.\tag{5}
\end{align*}
Now, let $A$ be the set of vertices $x_j$ with $j$ odd. Observe that for each $i=1, 2$, we have $A=A^{'}_i\cup A^{''}_i$ where
\begin{eqnarray*}
A^{'}_i&=&\Big\{x_j: x_jy_{j+d}\in E(F_i), ~ j, d\in \ZZ_{2\ell},~ j \equiv 1\pmod {2} \Big\},\\
A^{''}_i&=&\Big\{x_j: x_jx_{j \pm d }\in E(F_i), ~ j\in \ZZ_{2\ell}, d\in L, ~ j \equiv 1 \pmod{2} \Big\}.
\end{eqnarray*}
Similarly, let $B$ be the set of vertices $y_j$ with $j $ even. Then for each $i=1, 2$, we find that $B=B^{'}_i\cup B^{''}_i$ where 
\begin{eqnarray*}
B^{'}_i&=&\Big\{y_j: x_jy_{j+d}\in E(F_i),~ j, d\in \ZZ_{2\ell}, ~j \equiv 0 \pmod {2} \Big\},\\
B^{''}_i&=&\Big\{y_j: y_jy_{j\pm d}\in E(F_i), ~j\in \ZZ_{2\ell}, d\in L, ~j \equiv 0 \pmod{2} \Big\}.
\end{eqnarray*}

Observe that $|A|=|A^{'}_i|+|A^{''}_i|$ and $|B|=|B^{'}_i|+|B^{''}_i|$. Furthermore, for each $i=1,2 $, 
\begin{eqnarray*}
|A^{'}_{i}|&=&\varepsilon^{o}_{i}+\omega^{o}_{i}\quad \mbox{and}\\ 
|B^{'}_{i}|&=&\varepsilon^{e}_{i}+\omega^{o}_{i}.
\end{eqnarray*}
Note that the number of vertices $x_j$ with $j$ odd that are covered by edges of even left pure length in $F_i$, for each $i=1,2$, is even. Hence
$$|A^{''}_{i}|\equiv \lambda_{i}\pmod{2}.$$

Similarly, the number of vertices $y_j$ with $j$ even that are covered by edges of even right pure length in $F_i$, for each $i=1,2$, is even. Thus
$$
|B^{''}_{i}|\equiv \rho_{i}\pmod{2}.
$$
Consequently, we have 
\begin{align*}\label{EQA 6}
|A^{'}_1|+|A^{''}_1|&\equiv \varepsilon^{o}_{1}+\omega^{o}_{1}+\lambda_{1} \equiv 1 \pmod{2},\\
|B^{'}_1|+|B^{''}_1|&\equiv\varepsilon^{e}_{1}+\omega^{o}_{1}+\rho_{1}\equiv 1 \pmod{2},\tag{6}\\
|A^{'}_2|+|A^{''}_2|&\equiv\varepsilon^{o}_{2}+\omega^{o}_{2}+\lambda_{2}\equiv 1 \pmod{2},~\mbox{and}\\
|B^{'}_2|+|B^{''}_2|&\equiv \varepsilon^{e}_{2}+\omega^{o}_{2}+\rho_{2}\equiv 1 \pmod{2}.
\end{align*}
Adding up the equations in \eqref{EQA 6} and using \eqref{EQS 5}, we conclude that 
$$ \varepsilon^{o}_{1}+\varepsilon^{e}_{1}+ \varepsilon^{o}_{2}+\varepsilon^{e}_{2}\equiv 0 \pmod{2}.$$
On the other hand, the number of even mixed differences in $F_1$ and $F_2$ is jointly
$$ \varepsilon^{o}_{1}+\varepsilon^{e}_{1}+ \varepsilon^{o}_{2}+\varepsilon^{e}_{2}=\ell \equiv 1 \pmod{2},$$
which is a contradiction.

\end{proof}
\section*{Acknowledgements}
This research was carried out during the first author's visit in the Department of Mathematics and Statistics, University of Ottawa, as a Visiting Student Researcher. She gratefully acknowledges that funding for her research stay was provided by the Iranian Ministry of Science, Research and Technology and also Professor Mateja \v{S}ajna from University of Ottawa. The second author gratefully acknowledges support by the Natural Sciences and Engineering Research Council of Canada (NSERC), Discovery Grant RGPIN-2016-04798.

\end{document}